\documentclass{article}
\usepackage{fancyhdr,graphicx}
\usepackage{amsfonts}
\usepackage{amssymb}
\usepackage{amsthm}
\usepackage{newlfont}
\usepackage{amsmath}
\DeclareMathOperator{\corank}{corank}
\usepackage{multicol}
\usepackage{xcolor}
\usepackage{blkarray}
\usepackage{amscd}
\usepackage{subcaption}
\usepackage[utf8]{inputenc}

\newcommand{\JW}[1]{f_{1}}
\newcommand{\coeff}[2]{\coefficient_{\in \JW{1}^{2}}}

\DeclareMathOperator{\coefficient}{coeff}

\newcommand{\db}[1]{\left(\left(1\right)\right)}

\newtheorem{theorem}{Theorem}[section]
\newtheorem{lemma}[theorem]{Lemma}

\newtheorem{definition}[theorem]{Definition}
\newtheorem{claim}[theorem]{Claim}

\newtheorem{conjecture}[theorem]{Conjecture}
\graphicspath{{figures/}}

\usepackage{psfrag}

\usepackage{mathrsfs}
\usepackage{subfig}
\usepackage{picinpar}

\begin{document}

\title
	{Proof of a conjecture of Fomichev and Karev}
	
	\author{Qi Yan\\
		{\small School of Mathematics and Statistics, Lanzhou University, P. R. China}\\
        Qingying Deng\footnote{Corresponding author.}\\
        {\small School of Mathematics and Computational Science, Xiangtan University, P. R. China}\\
		Xian'an Jin\\
		{\small School of Mathematical Sciences, Xiamen University, P. R. China}\\
		{\small{Email: yanq@lzu.edu.cn (Q. Yan), qingying@xtu.edu.cn (Q. Deng), xajin@xmu.edu.cn (X. Jin)}}
	}
	\date{}
	\maketitle

\begin{abstract}

We prove a conjecture of Fomichev and Karev [{European J. Combin.} 127 (2025) 104160] by showing the equality of two graph invariants: $\varphi$, defined via graph colorings, and $\psi$, derived from the $\mathfrak{sl}(2)$-weight system of its 2-dimensional irreducible representation.
\end{abstract}

\section{Introduction}
In this paper, we only consider simple graphs. The set of isomorphism classes of simple graphs on $n$ vertices is denoted $\mathbf{G}_n$ and $\mathbf{G} = \bigcup_{n\geq 0} \mathbf{G}_n$.

For graph $G\in \mathbf{G}$, we denote the set of its vertices by $V(G)$ and the set of its edges by $E(G)$. For subset $E' \subset E(G)$, we denote by $G|_{E'}$ the spanning subgraph of $G$ with the set of edges $E'$. By $\chi_3(G)$ we mean the number of proper colorings of the vertices of $G$ in three colors. Recall that a function $V(G)\to
\{1,\ldots,k\}$ is a (proper) coloring of graph $G$ if no two vertices connected by an edge are assigned the same value.

Fomichev and Karev \cite{Fomichev}
introduced the function \(\varphi : \mathbf{G} \to \mathbb{C}\) and proved that it satisfies the graph Chmutov--Varchenko relations at \(\frac{3}{8}\).

\begin{definition}[Function $\varphi$]\label{defn:phi}
    For $G\in \mathbf{G}$, we set
    $$\varphi(G) = 2^{-3|V(G)|}\sum_{E' \subset E(G)}(-2)^{|E'|} \chi_3(G|_{E'}).$$
\end{definition}

It is known (see, e.g., \cite{Chmutov}) that the value of the \( \mathfrak{sl}(2)\)-weight system on its 2-dimensional irreducible representation can be explicitly computed for any chord diagram. Moreover, this formula can be extended to graphs as follows \cite{Fomichev}:

\begin{definition}[Function $\psi$]\label{defn:psi}
For $G\in \mathbf{G}$, we define
	\begin{align*}
		\psi(G)=
		\frac{1}{2^{2|V(G)|}} \sum_{U\subset V(G)} \left(-\frac{1}{2}\right)^{|V(G)|-|U|} 2^{\corank(A(G|_U))},
	\end{align*}
	where
	\begin{itemize}
		\item the sum is taken over all possible subsets $U$ of the vertex set $V(G)$;
		\item $G|_U$ is the subgraph of $G$ induced by vertex subset $U$;
		\item $A(G)$ is the adjacency matrix of $G$ with coefficients in the field with two elements;
		\item $\corank$ is the corank of the matrix.
	\end{itemize}
\end{definition}

\begin{conjecture}[\cite{Fomichev}]\label{con1}
For $G\in \mathbf{G}$, we have $$\varphi(G)=\psi(G).$$
\end{conjecture}

In this paper, we prove Conjecture \ref{con1}.

\section{Proof of Conjecture \ref{con1}}

Let $G \in \mathbf{G}$. For a subset $U \subset V(G)$, we denote by $E(U, V(G) \setminus U)$ the set of edges with one endpoint in $U$ and the other in $V(G) \setminus U$. A graph is \emph{Eulerian} if all vertex
degrees are even.

\begin{lemma}[\cite{Fomichev}]\label{prop:alt1}
	For a graph $G\in \mathbf{G}$ with the set of vertices $V(G)$ and the set of edges $E(G)$, we have
	\[
		\varphi(G) = 2^{-3|V(G)|} \sum_{\substack{U \subset V(G)\\ G|_U \text{ is Eulerian}}} (-1)^{|E(U,V(G)\setminus U)|} 2^{|U|}.
	\]
\end{lemma}

\noindent\textbf{Proof of Conjecture \ref{con1}.}
We begin with the definition of $\psi$:
\[
\psi(G) = 2^{-2|V(G)|}\sum_{U \subset V(G)} \left(-\frac12\right)^{|V(G)|- |U|}  2^{\corank(A(G|_U))}.
\]
Since $2^{\corank(A(G|_U))} = \left\lvert \left\{ y \in \mathbf{F}_2^U : A(G|_U) y = \mathbf{0} \right\} \right\rvert$, where $\mathbf{F}_2^U$ denotes the vector space of column vectors over $\mathbf{F}_2$ (the field with two elements) indexed by vertices in $U$ and $\mathbf{0}$ denotes the zero vector, we obtain
\[
\psi(G) = 2^{-2|V(G)|}\sum_{U \subset V(G)} \left(-\frac12\right)^{|V(G)|- |U|}  \left\lvert \left\{ y \in \mathbf{F}_2^U : A(G|_U) y = \mathbf{0} \right\} \right\rvert.
\]

We introduce the following notation for any vector $x \in \mathbf{F}_2^{V(G)}$: $\operatorname{supp}(x) = \{ i \in V(G) : x_i = 1 \}$ and $S(x) = \{ i \in V(G) : (Ax)_i = 0 \}$. Note that $S(x) = V(G) \setminus \operatorname{supp}(Ax)$. Based on these notations, we state the following claim.

\begin{claim}\label{claim}
For any subset $U \subset V(G)$, we have
\[
\left\lvert \left\{ y \in \mathbf{F}_2^U : A(G|_U) y = \mathbf{0} \right\} \right\rvert = \left\lvert \left\{ x \in \mathbf{F}_2^{V(G)} : \operatorname{supp}(x) \subset U \subset S(x) \right\} \right\rvert.
\]
\end{claim}

\begin{proof}
We construct an explicit bijection between the two sets. Define a mapping $\Phi$ from $\{ y \in \mathbf{F}_2^U : A(G|_U) y = \mathbf{0} \}$ to $\{ x \in \mathbf{F}_2^{V(G)} : \operatorname{supp}(x) \subset U \subset S(x) \}$ as follows:

For any $y \in \mathbf{F}_2^U$ satisfying $A(G|_U) y = \mathbf{0}$, define $\Phi(y) = x \in \mathbf{F}_2^{V(G)}$ by
\[
x_i =
\begin{cases}
y_i & \text{if } i \in U, \\
0 & \text{if } i \notin U.
\end{cases}
\]

We now verify that $\Phi$ is well-defined and establishes a bijection.

(1) \textbf{Well-definedness:} For any such $y$, we have $\operatorname{supp}(x) \subset U$ by definition. To show $U \subset S(x)$, note that $A(G|_U)$ is the submatrix of $A$ with rows and columns indexed by $U$. The condition $A(G|_U) y = \mathbf{0}$ means that for all $i \in U$, we have $\sum_{j \in U} A_{ij} y_j = 0$. Since $x_j = y_j$ for $j \in U$ and $x_j = 0$ for $j \notin U$, it follows that for any $i \in U$:
\[
(A x)_i = \sum_{j \in V(G)} A_{ij} x_j = \sum_{j \in U} A_{ij} x_j + \sum_{j \notin U} A_{ij} x_j = \sum_{j \in U} A_{ij} x_j = \sum_{j \in U} A_{ij} y_j = 0. \tag{1}
\]
Thus, $(A x)_i = 0$ for all $i \in U$, so  $U \subset S(x)$.

(2) \textbf{Injectivity:} Suppose $\Phi(y) = \Phi(y')=x$ for $y, y' \in \mathbf{F}_2^U$ with $A(G|_U) y = A(G|_U) y' = \mathbf{0}$. Then by definition of $\Phi$, we have $y_i = y_i'=x_i$ for all $i \in U$, so $y = y'$.

(3) \textbf{Surjectivity:} For any $x \in \mathbf{F}_2^{V(G)}$ with $\operatorname{supp}(x) \subset U \subset S(x)$, define $y \in \mathbf{F}_2^U$ as the restriction of $x$ to $U$. Then $\Phi(y) = x$ by definition. Moreover, since $U \subset S(x)$, we have $(A x)_i = 0$ for all $i \in U$. By equation (1), this implies that $A(G|_U) y = \mathbf{0}$.

Therefore, $\Phi$ is a bijection between the two sets, and the claim follows.
\end{proof}

By Claim \ref{claim}, we can rewrite $\psi(G)$ as:
\[
\psi(G) = 2^{-2|V(G)|} \sum_{U \subset V(G)} \left(-\frac12\right)^{|V(G)|- |U|} \left\lvert \left\{ x \in \mathbf{F}_2^{V(G)} : \operatorname{supp}(x) \subset U \subset S(x) \right\} \right\rvert.
\]
Interchanging the order of summation, we now consider the sum over $x \in \mathbf{F}_2^{V(G)}$ first, followed by the sum over all subsets $U$ satisfying $\operatorname{supp}(x) \subset U \subset S(x)$:

\[
\psi(G) = 2^{-2|V(G)|} \sum_{x \in \mathbf{F}_2^{V(G)}} \sum_{\operatorname{supp}(x) \subset U \subset S(x)} \left(-\frac12\right)^{|V(G)|- |U|}.
\]

Fix a vector $x \in \mathbf{F}_2^{V(G)}$ and let $n = |V(G)|$, $s = |\operatorname{supp}(x)|$, and $m = |S(x)|$.
The subsets $U$ satisfying $\operatorname{supp}(x) \subset U \subset S(x)$ are precisely those obtained by taking $\operatorname{supp}(x)$ (of size $s$) and adding $t$ additional vertices from $S(x) \setminus \operatorname{supp}(x)$ (which has size $m - s$), where $0 \le t \le m - s$.

For fixed $t$, we have $|U| = s + t$, so $|V(G)| - |U| = n - s - t$. Then
\[
\begin{aligned}
\sum_{\operatorname{supp}(x) \subset U \subset S(x)} \left(-\frac12\right)^{|V(G)|- |U|}
&= \sum_{t=0}^{m-s} \binom{m-s}{t} \left(-\frac12\right)^{n- s -t} \\
&= \left(-\frac12\right)^{n-s} \sum_{t=0}^{m-s} \binom{m-s}{t} (-2)^t \\
&= (-1)^{m-s} \left(-\frac{1}{2}\right)^{n-s} \\
&= (-1)^{m-s} (-1)^{n-s}  \left(\frac12\right)^{n-s} \\
&= (-1)^{m+n-2s} 2^{-(n-s)}
\\
&= (-1)^{m+n} 2^{s-n}.
\end{aligned}
\]
Recall that \( m = |S(x)| = n - |\operatorname{supp}(Ax)| \), which implies that \( m + n = 2n - |\operatorname{supp}(Ax)| \).
As \( 2n \) is even, we obtain \( (-1)^{m+n} = (-1)^{|\operatorname{supp}(Ax)|} \).
Therefore,
\[
\begin{aligned}
\sum_{\operatorname{supp}(x) \subset U \subset S(x)} \left(-\frac12\right)^{|V(G)|- |U|}
= (-1)^{|\operatorname{supp}(Ax)|}  2^{s-n}
=(-1)^{|\operatorname{supp}(Ax)|} 2^{|\operatorname{supp}(x)| - n}.
\end{aligned}
\]			
Thus
\[
\begin{aligned}
\psi(G) &= 2^{-2n} \sum_{\substack{x \in \mathbf{F}_2^{V(G)} \\ \operatorname{supp}(x) \subset S(x)}} (-1)^{|\operatorname{supp}(A x)|}  2^{|\operatorname{supp}(x)| - n} \\
&= 2^{-3n} \sum_{\substack{x \in \mathbf{F}_2^{V(G)} \\ \operatorname{supp}(x) \subset S(x)}} (-1)^{|\operatorname{supp}(A x)|} 2^{|\operatorname{supp}(x)|}.
\end{aligned}
\]

There is a natural bijection between vectors $x \in \mathbf{F}_2^{V(G)}$ and subsets $U \subset V(G)$ given by $U = \operatorname{supp}(x)$. Under this correspondence, we observe that $2^{|\operatorname{supp}(x)|} = 2^{|U|}$ and $(-1)^{|\operatorname{supp}(A x)|} = (-1)^{|\operatorname{supp}(A \mathbf{1}_U)|}$, where $\mathbf{1}_U$ denotes the characteristic vector of $U$ (i.e., $(\mathbf{1}_U)_i = 1$ if $i \in U$ and $0$ otherwise). Furthermore, the condition $\operatorname{supp}(x) \subset S(x)$ translates to $U \subset S(U)$, where $S(U) = V(G) \setminus \operatorname{supp}(A \mathbf{1}_U)$.
Thus
\[
\psi(G) = 2^{-3n} \sum_{\substack{U \subset V(G) \\ U \subset S(U)}} (-1)^{|\operatorname{supp}(A \mathbf{1}_U)|} 2^{|U|}. \tag{2}
\]

For any vertex $i \in V(G)$ and any subset $U \subset V(G)$, let $N_G(i)$ denote the set of neighbors of $i$ in $G$ and define $\deg_U(i) = |N_G(i) \cap U|$. Then for all $i \in V(G)$ the vector $A \mathbf{1}_U$ has components $$(A \mathbf{1}_U)_i = \deg_U(i) \bmod 2.$$
Thus, we have
\[
|\operatorname{supp}(A \mathbf{1}_U)| = |\{ i \in V(G) : \deg_U(i) \equiv 1 \pmod{2} \}|
\]
and
\[
S(U) = \{ i \in V(G) : \deg_U(i) \equiv 0 \pmod{2} \}.
\]
The condition $U \subset S(U)$ is equivalent to $\deg_U(i) \equiv 0 \pmod{2}$ for all $i \in U$, implying that every vertex in the induced subgraph $G|_U$ has even degree, i.e., $G|_U$ is Eulerian.

Now, consider the sum of $\deg_U(i)$ over all vertices $i \in V(G)$. We observe that
\[
\begin{aligned}
|\operatorname{supp}(A \mathbf{1}_U)| &= |\{ i \in V(G) : \deg_U(i) \equiv 1 \pmod{2} \}| \\
&= \sum_{i \in V(G)} (\deg_U(i) \bmod 2) \\
&\equiv \sum_{i \in V(G)} \deg_U(i) \pmod{2}.
\end{aligned} \tag{3}
\]

On the other hand, we can compute the sum explicitly by considering edge contributions:
\[
\sum_{i \in V(G)} \deg_U(i) = 2|E(U)| + |E(U, V(G) \setminus U)|,
\]
since each edge within $U$ contributes $2$ to the sum, while each edge between $U$ and its complement contributes $1$. Reducing modulo $2$ yields:
\[
\sum_{i \in V(G)} \deg_U(i) \equiv |E(U, V(G) \setminus U)| \pmod{2}. \tag{4}
\]
From equations (3) and (4), we obtain
\[
|\operatorname{supp}(A \mathbf{1}_U)| \equiv |E(U, V(G) \setminus U)| \pmod{2},
\]
which implies
\[
(-1)^{|\operatorname{supp}(A \mathbf{1}_U)|} = (-1)^{|E(U, V(G) \setminus U)|}.
\]

Therefore, we can rewrite the sum as:
\[
\sum_{\substack{U \subset V(G) \\ U \subset S(U)}} (-1)^{|\operatorname{supp}(A \mathbf{1}_U)|} 2^{|U|} = \sum_{\substack{U \subset V(G) \\ G|_U \text{ is Eulerian}}} (-1)^{|E(U, V(G) \setminus U)|} 2^{|U|}.
\]
Combining this with equations (2), we conclude
\[
\psi(G) = 2^{-3n} \sum_{\substack{U \subset V(G) \\ G|_U \text{ is Eulerian}}} (-1)^{|E(U, V(G) \setminus U)|} 2^{|U|}.
\]
By Lemma \ref{prop:alt1}, the right-hand side is exactly $\varphi(G)$, thus proving that $\psi(G) = \varphi(G)$.

\section*{Acknowledgements}
This work is supported by NSFC (Nos. 12471326, 12571379), and partially supported
by the the 111 Project (No. D23017), the Excellent Youth Project of Hunan Provincial
Department of Education, P. R. China (No. 23B0117).

\end{document}